\documentclass[preprint]{imsart}

\RequirePackage{amsthm,amsmath,amsfonts,amssymb}
\RequirePackage[numbers,sort&compress]{natbib}
\RequirePackage[colorlinks,citecolor=blue,urlcolor=blue]{hyperref}
\RequirePackage{graphicx}

\startlocaldefs
\theoremstyle{plain}

\newtheorem{theorem}{Theorem}[section]
\newtheorem{lemma}[theorem]{Lemma}

\theoremstyle{definition}
\newtheorem{definition}[theorem]{Definition}
\newtheorem*{example}{Example}

\endlocaldefs
\usepackage[margin = 1in]{geometry}
\usepackage{amsmath}
\usepackage{amssymb}
\usepackage{multirow}
\usepackage{amsthm}  
\usepackage{graphicx}
\usepackage{subcaption}
\usepackage{float}  
\usepackage[ruled,vlined]{algorithm2e}
\SetKwInOut{KwIn}{Input}  
\SetKwInOut{KwOut}{Output}
\SetAlFnt{\small}
\SetAlCapFnt{\small}
\SetAlCapNameFnt{\small}
\usepackage{comment}
\newtheorem{remark}{Remark}
\newtheorem{corollary}{Corollary}
\newtheorem{assumption}{Assumption}
\makeatletter
\newcommand{\algocf@resetlines}{\relax}
\makeatother

\begin{document}

\begin{frontmatter}
\title{NON ASYMPTOTIC MIXING TIME ANALYSIS \\  OF NON-REVERSIBLE MARKOV CHAINS}


\author{\fnms{Muhammad Abdullah}~\snm{Naeem}%
  \thanks{For any suggestions, please contact me via 
  \texttt{muhammadabdullahnaeem@gmail.com} or 
  \texttt{abdullah\_naeem\_waraich@outlook.com}.}}

\begin{abstract}
We introduce a unified operator-theoretic framework for analyzing mixing times of finite-state ergodic Markov chains that applies to both reversible and non-reversible dynamics. The central object in our analysis is the projected transition operator $PU_{\perp 1}$, where $P$ is the transition kernel and $U_{\perp 1}$ is orthogonal projection onto mean-zero subspace in $\ell^{2}(\pi)$, where $\pi$ is the stationary distribution.  We show that explicitly computable matrix norms of $(PU_{\perp 1})^k$ gives non-asymptotic mixing times/distance to stationarity, and bound autocorrelations at lag $k$. We establish, for the first time, submultiplicativity of pointwise chi-squared divergence in the general non-reversible case. Leveraging upon non-asymptotic decay of powers of stable transition matrices in \cite{naeem2023spectral}, we provide for all times $\chi^{2}(k)$ bounds based on the spectrum of $PU_{\perp 1}$, i.e., magnitude of its distinct non-zero eigenvalues, discrepancy between their algebraic and geometric multiplicities, condition number of a similarity transform, and constant coming from smallest atom of stationary distribution(all  scientifically computable). Furthermore, for diagonalizable $PU_{\perp 1}$, we provide explict constants satisfying hypocoercivity phenomenon\cite{villani2006hypocoercivity} for discrete time Markov Chains. Our framework enables direct computation of convergence bounds for challenging non-reversible chains, including momentum-based samplers for V-shaped distributions. We provide the sharpest known bounds for non-reversible walk on triangle( \cite{choi2020metropolis,montenegro2006mathematical}). Our results combined with simple regression reveals a fundamental insight into momentum samplers: although for uniform distributions, $n\log{n}$ iterations suffice for $\chi^{2}$ mixing, for V-shaped distributions they remain diffusive as $n^{1.969}\log{n^{1.956}}$ iterations are sufficient. The framework also provides new insights into autocorrelation structure and concentration of ergodic averages, by showing that for ergodic chains relaxation times introduced in \cite{chatterjee2023spectral} $\tau_{rel}=\|\sum_{k=0}^{\infty}P^{k}U_{\perp 1}\|_{\ell^{2}(\pi)}$. 
\end{abstract}

\end{frontmatter}

\section{Introduction}
Markov chains are central to probability, statistics and computational science: their long-time behaviour underpins Markov Chain Monte Carlo (MCMC) algorithms, network dynamics, and stochastic models in physics and biology. For reversible kernels the spectral theory of self-adjoint operators gives powerful, quantitative control of convergence to stationarity via spectral gaps; for non-reversible kernels, however, the situation is substantially more delicate. Non-normality can produce initial transient growth of operator norms and existing asymptotic analysis based on Gelfands' formula hides the mechanisms controlling finite-time convergence. Furthermore, classical spectral gap-based bounds can be loose or misleading. Understanding non-reversible dynamics is therefore both theoretically challenging and practically important given the recent interest in non-reversible samplers and lifted dynamics that can outperform reversible schemes in many settings \citep{diaconis2000analysis, eberle2025convergence}.
\paragraph{Motivation.}
Non-reversible dynamics appear naturally in dynamical systems, statistical physics and have been deliberately introduced in Markov Chain Monte Carlo (MCMC) to reduce random walk behaviour \cite{diaconis2000analysis,chen1999lifting}.  
Early work by \citet{Neal2004} showed that breaking detailed balance can strictly improve the asymptotic variance of MCMC estimators.  
Subsequent developments in ``lifting'' and ``momentum'' samplers (see the survey \citep{Vucelja2014}) introduced auxiliary variables to induce directed motion. Empirical and theoretical analyses confirm that such schemes can outperform reversible ones in many regimes. 
  
\paragraph{Main idea of the paper.}
To understand convergence rates of a Markov chain to its stationary distribution, it is instructive to begin with convergence in the total variation metric. 
For $\epsilon>0$, define the total variation mixing time as
\[
\tau_{\mathrm{TV}}(\epsilon)
:= \inf \{\,k > 0 : \|p^{k}(x,\cdot) - \pi\|_{\mathrm{TV}} < \epsilon,\ \forall x \in \mathcal{X}\,\}.
\]
Intuitively, this quantifies how fast all rows of $P^{k}$ approach the stationary distribution vector $\pi \in \mathbb{R}^{n}$. In the finite-state setting, exact mixing behavior can be characterized through matrix norms. 
Consider the $P$-invariant subspace orthogonal to constant functions in $\ell^{2}(\pi)$,
\[
V_{\perp 1} = \{ f \in \ell^{2}(\pi) : \langle f, \mathbf{1} \rangle_{\pi} = 0 \},
\]
and let $U_{\perp 1} := I_{n} - \mathbf{1}\pi^{\top}$ be the projection onto $V_{\perp 1}$.
Then
\[
P^{k} U_{\perp 1} = P^{k} - \mathbf{1}\pi^{\top},
\]
and the $\ell_{1}$-norm of each row of $P^{k} U_{\perp 1}$ equals
$2\|p^{k}(x,\cdot) - \pi\|_{\mathrm{TV}}$ for the corresponding state $x$.
Hence total variation convergence can equivalently be written as
\[
\tau_{\mathrm{TV}}(\epsilon)
 = \inf\{\,k : \|P^{k}U_{\perp 1}\|_{\infty} < 2\epsilon\,\}.
\]

More generally, the decay rate of the projected operator $\|P^{k}U_{\perp 1}\|$
dictates the chain’s convergence to equilibrium:
the $\ell^{2}(\pi)$-operator norm controls $\chi^{2}$ convergence, while
the $\ell_{\infty}$-operator norm controls total variation convergence.
After introducing suitable similarity transforms, we show that these operator norms can be computed explicitly, yielding non-asymptotic convergence rates to stationarity.

\paragraph{Our contribution.}
This paper develops an operator-theoretic framework that provides explicit, computable bounds for the convergence of finite non-reversible Markov chains.  
The central object controlling finite-time convergence in $\ell^{2}(\pi)$ and $\chi^{2}$ sense is $\|P^{k} U_{\perp 1}\|_{\ell^{2}(\pi)}$, by using similarity transformation we reduce convergence analysis to spectral norm of $\|D_{\pi}^{\tfrac{1}{2}}P^{k}U_{\perp 1}D_{\pi}^{-\tfrac{1}{2}}\|_{\ell^{2}(\mathbb{R}^{n})}$, where $D_{\pi}=\mathrm{diag}(\pi_{1},\ldots,\pi_{n})$(assuming strictly positive support).
Within this setting we derive:
\begin{enumerate}
  \item A direct identification between $\ell^{2}(\pi)$-distance to stationarity at time $k$,spectral norm of projected similar matrix and and second largest singular value of transition kernel:
  \[\|p^{k}-\pi\|_{\ell^{2}(\pi)}=\sigma_{2}(P^{k})=\|D_{\pi}^{\tfrac{1}{2}}P^{k}U_{\perp 1}D_{\pi}^{-\tfrac{1}{2}}\|_{\ell^{2}(\mathbb{R}^{n})}.\]
  \item A \emph{pointwise} $\chi^{2}$ submultiplicativity inequality 
        \[
        \chi^{2}_{i}(k+t)\leq \sigma_{2}^{2}(P^{t})\,\chi^{2}_{i}(k),
        \]
        giving sharp finite-time control of convergence;
  \item Convergence rate for all $k$ via \emph{spectral properties of a single matrix} $D_{\pi}^{\tfrac{1}{2}}PU_{\perp 1}D_{\pi}^{-\tfrac{1}{2}}$: If $P \in \mathbb{R}^{n \times n}$ and $\{\lambda_{i}\}_{i \in [K]}$ (where $K\leq n-1$) be distinct non-zero eigenvalues of $PU_{\perp 1}$. \emph{Discrepancy} of eigenvalue $\lambda_{i}$ is defined as the difference between its' \emph{algebraic multiplicty} and \emph{geometric multiplicity} and denoted by $D_{\lambda_{i}}$. Let $D_{\pi}^{\frac{1}{2}}PU_{\perp1}D_{\pi}^{-\frac{1}{2}}=SJ_{\perp1}S^{-1}$ be its Jordan canonical form decomposition. Then for all $k \in \mathbb{N}$:
    \begin{align}
        & \nonumber\chi^{2}(k) \leq \left(\frac{1-\pi_{\min}}{\pi_{\min}}\right) \kappa(S)^2 \max_{1 \leq i \leq K} \left[ k^{2D_{\lambda_{i}}} |\lambda_{i}|^{2k} \left( \frac{1-|\lambda_{i}|}{1-|\lambda_{i}|^{D_{\lambda_{i}}+1}}\right)^{2} \right],
    \end{align}
    where $\kappa(S) = \|S\|_{\ell^{2}(\mathbb{R}^{n})}\|S^{-1}\|_{\ell^{2}(\mathbb{R}^{n})}$ is the condition number of the similarity transformation and $\pi_{\min}$ is the probability mass of smallest atom in $\pi$.       
  \item Explicit \emph{hypocoercivity constants} for diagonalizable dynamics:
  \[\|p^{k}-\pi\|_{\ell^{2}(\pi)} \leq \kappa(S)\rho^{k}(D_{\pi}^{\frac{1}{2}}PU_{\perp1}D_{\pi}^{-\frac{1}{2}}),\] where $\rho(D_{\pi}^{\frac{1}{2}}PU_{\perp1}D_{\pi}^{-\frac{1}{2}})$ is used to denote \emph{spectral radius} of projected and conjugated transition matrix and is strictly less than $1$ under ergodicity assumption.
  \item For ergodic Markov chains \emph{relaxation time} recently defined by \cite{chatterjee2023spectral} to compute the concentration of ergodic averages, fundamentally captures the cumulative effect of autocorrelations across all time lags $k$ and can be explicitly computed as:
   \begin{align}
       & \nonumber \Bigg\|D_{\pi}^{\frac{1}{2}}\Big(\sum_{k=0}^{\infty}P^{k}\Big)U_{\perp 1}D_{\pi}^{-\frac{1}{2}}\Bigg\|_{\ell^{2}(\mathbb{R}^{n})}=\tau_{rel}.
   \end{align}
\end{enumerate}
These results yield finite-time, non-asymptotic mixing bounds with constants that can be computed directly from the transition matrix. Consequently we were able to derive: best known bound for $\chi^{2}$ distance to stationarity for three-state non-reversible Markov Chain from \cite{choi2017analysis}. Although $n\log{n}$ samples suffices for $\chi^{2}$ convergence of momentum based sampler for uniform distribution (as shown in \cite{diaconis2000analysis}), \emph{momentum based sampler is diffusive for \emph{V-shaped distribution}} on $n-$ points of the form $\pi(x) \propto 2|x-\tfrac{n}{2}|+2$ as shown in Table \ref{tab:asymp_conv_rates}.   \begin{remark}
    Momentum based samplers for uniform distribution and $\mathrm{V}$- shaped distribution are diagonalizable, so spectral radius of projected conjugated matrix, combined with $\pi_{\min}$ estimates leads to $\chi^{2}(k)$ bounds in Table \ref{tab:asymp_conv_rates} \emph{modulo condition number of similarity}.  
\end{remark}   

\begin{table}[H]
\centering
\begin{tabular}{|l|c|}
\hline
\textbf{Markov Chains} & \textbf{Convergence Rates} \\ 
\hline
\textit{Momentum based sampler for uniform distribution on $n$ points with $\theta=\frac{1}{n}$} & $\chi^{2}(k) \leq (2n-1)\mathrm{O}\Big(1-\frac{1}{n}\Big)^{2k}$ \\ 
\hline
\textit{Momentum based sampler for $\pi(x)\propto 2|x-\frac{n}{2}|+2$, $n$ even and $\theta=\frac{1}{n}$ } & $\chi^{2}(k) \leq (Cn^{\beta}-1) \mathrm{O}\Big(1-\frac{c}{n^{\alpha}}\Big)^{2k}$ \\ & $c=6.35,\alpha=1.962, C=0.638,\beta=1.956$ \\ 
\hline
\textit{Momentum based sampler for $\pi(x)\propto 2|x-\frac{n}{2}|+2$, $n$ odd and $\theta=\frac{1}{n}$ } & $\chi^{2}(k) \leq (Cn^{\beta}-1)\mathrm{O}\Big(1-\frac{c}{n^{\alpha}}\Big)^{2k}$ \\ & $c=9.83,\alpha=1.969,C=1.290,\beta=1.956$ \\ 
\hline
\textit{Three state non-reversible chain(no spectral gap)} & $\|p^{k}-\pi\|_{TV} \leq \mathrm{O}\Big(\sqrt{\tfrac{1}{2}}^{k}\Big)$ \\ 
\hline
\end{tabular}
\caption{Iteration and dimensionality dependent convergence rates of different Markov Chains.}
\label{tab:asymp_conv_rates}
\end{table}

\paragraph{Existing analyses on non-reversible chains}
While there exists substantial literature on non-reversible MCMC and lifting methods 
\cite{Neal2004,Vucelja2014}, 
most results concern asymptotic variance reduction or specific model families. Initial progress in understanding mixing times of non-reversible chains was made through various notions of reversibilization and alternative spectral gaps \citep{choi2017analysis, fill1991eigenvalue, paulin2015concentration,HuangMao2017}. These approaches, while valuable, often rely on heuristics for asymptotic convergence rates by constructing reversible versions of non-reversible chains or employing techniques that may not fully leverage the specific structure of non-reversible dynamics. A fundamental challenge has been the lack of computable methods for determining key quantities such as $\|p^k - \pi\|_{\ell^2(\pi)}$ for all times $k$ which controls the decay of autocorrelations and statistical efficiency of MCMC estimators. More recently \citep{ChoiPatie2020} offer valuable insight but provided bounds on $\|p^k - \pi\|_{\ell^2(\pi)}$ are only valid for diagonalizable dynamics as full spectral structure of non-Hermitian operators based on discrepancy between algebraic and geometric multiplicities of eigenvalues is not considered.
\paragraph{Fundamental limitaton}
In summary: most classical bounds fail to describe the transient regime of non-reversible chains.
Let $\lambda_{*} := \sup\{|\lambda| : \lambda \in \mathrm{spec}(P)\setminus\{1\}\}$($\mathrm{spec}(P)$ denotes the spectrum of $P$ i.e., its' eigenvalues) 
be the second largest eigenvalue in modulus, and let $\sigma_{2}(P)$ 
denote the second largest singular value of $P$ in $\ell^{2}(\pi)$ sense.
A common but naive approach bounds $\|P^{k}U_{\perp 1}\|_{\ell^{2}(\pi)}$
by $\lambda_{*}^{k}$ or $\sigma_{2}^{k}(P)$.
However, for non-self-adjoint $P$, the two quantities can differ drastically from $\|P^{k}U_{\perp 1}\|_{\ell^{2}(\pi)}$.
While ergodicity ensures $\rho(PU_{\perp 1})<1$, i.e., \emph{spectral radius} of projected transition kernel is strictly less than one and hence
$\|P^{k}U_{\perp 1}\|_{\ell^{2}(\pi)} \to 0$ as $k \to \infty$,
it need not be true that
$\|P^{k}U_{\perp 1}\|_{\ell^{2}(\pi)} = \|PU_{\perp 1}\|_{\ell^{2}(\pi)}^{k}$,
and indeed $\|P^{k}U_{\perp 1}\|_{\ell^{2}(\pi)}$ may be $1$ for small~$k$.
Thus bounds of the form $\|P^{k}U_{\perp 1}\|_{\ell^{2}(\pi)} \leq \sigma_{2}^{k}(P)$, in case of no spectral gap but ergodic chain, is essentially saying $\|P^{k}U_{\perp 1}\|_{\ell^{2}(\pi)}\leq 1$ as you will see such an example in remark \ref{rmk:three_state_non_rev}. Moreover, any bound of the form $\|P^{k}U_{\perp 1}\|_{\ell^{2}(\pi)} \leq  \kappa(S)\lambda_{*}^{k}$ is invalid for non-diagonalizable dynamics: consider the $3\times3$ transition matrix 
\( 
P=\begin{pmatrix}
0 & 1 & 0\\
0 & 0 & 1\\
\lambda^{2} & -(2\lambda+\lambda^{2}) & 1+2\lambda
\end{pmatrix}
\)
with $\lambda=-\tfrac14$, which has eigenvalues $1,\lambda,\lambda$. Moreover, eigenvalue $\lambda$ has algebraic multiplicity of ~$2$ and geometric multiplicity~$1$.  
Let $U_{\perp1}=I-\mathbf{1}\pi^{\top}$ denote the projection onto the orthogonal complement of the stationary distribution~$\pi=[0.04,0.32,0.64]$. So $\rho(PU_{\perp 1})=\frac{1}{4}=:\lambda_{*}$ and numerical computation gives $\|P^{k}U_{\perp1}\|_{\ell^{2}(\pi)}=
(0.5546,\,0.2374,\,0.0857,\,0.0282,\,0.0087,\,\dots),
\qquad k=1,\ldots,10.
$, deeming void exponential bound based on spectral radius for all $k$ and verifying our non-asymptotic bound based on discrepancy of $1$ for eigenvalue $\lambda$: $\|P^{k}U_{\perp 1}\|_{\ell^{2}(\pi)}\leq \kappa(S)k|\lambda|^{k}.$

\paragraph{Paper Structure}
Section \ref{sec:prelim} introduces the preliminary concepts and develop operator-theoretic framework necessary for analysis of non-reversible Markov Chains. In Section \ref{sec:main_theory_fwk} presents our main theoretical framework, including characterizations of various mixing times, computational methods, and connections to relaxation times. Section \ref{sec:case_studies} demonstrates the application of our framework through several case studies, highlighting its computational advantages and concluding the results in Table \ref{tab:asymp_conv_rates}. 

\section{Preliminaries and Operator Setup}
\label{sec:prelim}
For a distribution $\pi$ on an $n-$ dimensional space, there often exists different transition kernels(Markov chains) that converge to the steady state distribution $\pi$. We are interested in analyzing how spectral properties of kernel $P$ affect the mixing time of the underlying chain, our analysis works regardless of reversibility assumption. As the aim of MCMC algorithms is convergence of the chain to stationary distribution, irrespective of initial conditions, this requires us to ensure:  
\begin{assumption}
 Stationary distribution has \emph{strictly positive support} on all individual atoms in the underlying space: i.e., $\pi_{i}>0, \forall i \in [n]$ and transition kernel $P$ is \emph{irreducible} and \emph{aperiodic} with $\pi$ being its' stationary distribution.
\end{assumption}
We can define a Hilbert space $\ell^{2}(\pi)$ of $n$-dimensional vector valued functions $f=[f_{1},\ldots,f_{n}]$ satisfying $\|f\|_{\ell^{2}(\pi)}:=\sqrt{\sum_{i=1}^{n}|f_{i}|^{2}\pi_{i}}<\infty$, for $f,g \in \ell^{2}(\pi)$ inner product is defined as $\langle f,g\rangle_{\pi}:=\sum_{i=1}^{n}f_{i}g_{i}\pi_{i}$.  Given $p(i,j)$, probability of going from state $i$ to $j$ for all $i,j \in [n]$ we can generate a matrix $P \in \mathbb{R}^{n \times n}$, its' element at row $i$ and column $j$ being $p(i,j)$. Let $\pi=[\pi_{1},\ldots,\pi_{n}]$ be the column vector in $\mathbb{R}^{n}$ corresponding to stationary distribution then $\pi^{T}P=\pi^{T}$: i.e., \emph{$\pi$ is $P$-invariant}. Keep in mind that $\ell^{2}(\mathbb{R}^{n})$ is defined similarly to $\ell^{2}(\pi)$, enforcing that $\pi_{i}=1$ for all $i \in [n]$. We denote column vector of ones by $\mathbf{1} \in \mathbb{R}^{n}$ and notice that $P\mathbf{1}=\mathbf{1}$: i.e., \emph{space of constant functions in $\ell^{2}(\pi)$ and $\ell^{2}(\mathbb{R}^{n})$ space is invariant under $P$}.         is a \emph{contraction} on $\ell^{2}(\pi)$ i.e., $\sup_{f:\|f\|_{\ell^{2}(\pi)}=1}\|Pf\|_{\ell^{2}(\pi)} \leq 1$,
\[
Pf=
\begin{bmatrix}
        p(1,1) & p(1,2) & \ldots &  p(1,n) \\    
        p(2,1) & p(2,2) & \ldots &  p(2,n) \\ 
        \vdots & \vdots & \ddots & \vdots  \\ 
        p(n,1) & p(n,2) & \ldots & p(n,n)   
\end{bmatrix}
\begin{bmatrix}
    f_{1} \\ f_{2} \\ \vdots \\ f_{n}
\end{bmatrix}
=:
\begin{bmatrix}
    Pf(1) \\ Pf(2) \\ \vdots \\ Pf(n),
\end{bmatrix}
\]
where $Pf(i)=\sum_{j=1}^{n}f_{j}p(i,j)$ for each $i \in [n]$ and
\begin{align}
\|Pf\|_{\ell^{2}(\pi)}^{2} & \nonumber =\sum_{i=1}^{n}|Pf(i)|^{2}\pi_{i}\\&\nonumber=\sum_{i=1}^{n}|\sum_{j=1}^{n}f_{j}p(i,j)|^{2}\pi_{i} \\ & \label{eq:jensen}\leq \sum_{i=1}^{n}\sum_{j=1}^{n}|f_{j}|^{2}p(i,j)\pi_{i}\\ & \nonumber=\sum_{i=1}^{n}\sum_{j=1}^{n}|f_{j}|^{2}p(i,j)\pi_{i} \\ & \nonumber=\sum_{j=1}^{n}|f_{j}|^{2}\sum_{i=1}^{n}p(i,j)\pi_{i} \\ & \label{eq:invar} =\sum_{j=1}^{n}|f_{j}|^{2}\pi_{j} \\ & \nonumber :=\|f\|_{\ell^{2}(\pi)}, 
\end{align}
where inequality \ref{eq:jensen} follows from Jensens' inequality and equality \ref{eq:invar} follows from the $\pi^{\top}P=P$. Given any $f,g \in \ell^{2}(\pi)$ adjoint $P^{*}$ of transition kernel in $\ell^{2}(\pi)$ is defined as to satisfy:  
\begin{align}
    \langle Pf,g\rangle_{\pi} = \langle f,P^{*}g\rangle_{\pi} \quad \forall f,g \in \ell^{2}(\pi)
\end{align}
Adjoint operator $P^{*}$ generates a time-reversed transition kernel with stationary distribution $\pi$ with probability of going from state $i$ to $j$ being $p^{^{*}}(j,i)=\frac{p(i,j)\pi_i}{\pi_j}$ and $P^{*}\mathbf{1}=\mathbf{1}$. If $P^{T}$ is the transpose of the matrix $P$ in euclidean sense, under the hypothesis of strictly positive support of stationary distribution, $P^{*}$ as a bounded operator on $\ell^{2}(\pi)$ has matrix representation:   
\begin{align}
    \label{eq:P_l2_adj}
    P^{*}=D_{\pi}^{-1}P^{T}D_{\pi}, \hspace{5pt}\textit{where} \quad 
    D_{\pi} &=\mathrm{diag}(\pi_{1},\ldots,\pi_{n}).
\end{align}
If $P$ is \emph{self-adjoint} w.r.t $\ell^{2}(\pi)$  then $P=D_{\pi}^{-1}P^{^{T}}D_{\pi}$. Under strictly positive support hypothesis, if $P$ is \emph{self-adjoint} w.r.t $\ell^{2}(\pi)$ then $P$ satisfies \emph{Detailed Balance}/\emph{reversibility} condition i.e., for all $i,j \in [n]: \hspace{3pt} \pi_{i}p(i,j)=\pi_{j}p(j,i)$, so reversibility, detailed balance and self-adjoint will be used interchangeably in this paper.

\begin{remark}[Properties of the projection operator $U_{\perp 1}$]
Define $U_{\perp 1} := I_{n} - \mathbf{1}\pi^{\top}$,
\begin{itemize}
    \item $U_{\perp 1}$ is a projection matrix in the algebraic sense, satisfying:
        \[
            U_{\perp 1}^{2} = (I_{n} - \mathbf{1}\pi^{\top})^{2}
            = I_{n} - 2\mathbf{1}\pi^{\top} + \mathbf{1}(\pi^{\top}\mathbf{1})\pi^{\top}
            = I_{n} - \mathbf{1}\pi^{\top} = U_{\perp 1},
        \]
    \item $U_{\perp 1}$ projects onto an $(n-1)$-dimensional
    subspace of $\mathbb{R}^{n}$:
        \[
            \mathrm{trace}(U_{\perp 1})= \mathrm{trace}(I_{n})-\mathrm{trace}(1\pi^{\top})=n-\sum_{i=1}^{n}\pi_{i}=(n-1).   
        \]
    
    \item Image of $U_{\perp 1}$ consists of all mean-zero vectors with respect to~$\pi$: 
    
        \[
            \forall{f} \in \mathbb{R}^{n}:U_{\perp 1}f = f - \pi(f)\mathbf{1},
            \qquad 
            \text{where } \pi(f) := \sum_{i=1}^{n}\pi_i f_i \implies \langle U_{\perp 1}f,1\rangle_{\pi}=0.
        \]
    Note: Conventionally, centered vector valued $f$ i.e., $f-\pi(f)\mathbf{1}$ is denoted by $f-\pi(f)$.  
    \item $U_{\perp 1}$ is self-adjoint with respect to the $\pi$-weighted inner product $\langle f,g\rangle_{\pi} = f^{\top}D_{\pi}g$: 
    \[
        U_{\perp 1}^{*}
        := D_{\pi}^{-1}U_{\perp 1}^{\top}D_{\pi}
        = D_{\pi}^{-1}(I - \pi\mathbf{1}^{\top})D_{\pi}
        = I - \mathbf{1}\pi^{\top}
        = U_{\perp 1}.
    \]
    Hence $U_{\perp 1}$ acts as the orthogonal projection onto the mean-zero subspace
    \[
        V_{\perp 1} := \{\,f \in \ell^{2}(\pi) : \langle f, \mathbf{1}\rangle_{\pi} = 0\,\}.
    \]
    Note: $U_{\perp 1}$ is not symmetric in euclidean sense, unless $\pi$ is uniform.
    \item Commutation with transition kernel:
    If $P$ is a Markov transition kernel satisfying $\pi^{\top}P = \pi^{\top}$,
    then mean-zero subspace $V_{\perp 1}$ is $P$-invariant, and
        \[
            PU_{\perp 1} = U_{\perp 1}P.
        \]
    Consequently,
        \[
            (PU_{\perp 1})^{k} = P^{k}U_{\perp 1}, \qquad \forall\, k \ge 1.
        \]
    This identity shows that projecting before or after $k$ transitions yields the same result on $\ell^{2}(\pi)$ and $\mathbb{R}^{n}$.    
\end{itemize}
In summary, $U_{\perp 1}$ is an idempotent matrix of rank $n-1$ in $\mathbb{R}^{n}$,
    and the orthogonal projector onto $V_{\perp 1}$ in $\ell^{2}(\pi)$.
\end{remark}
\begin{remark}
\label{rmk: sing_and_mat_norms}
It is important to refresh singular values and matrix norm of the transition kernel for all powers $k$, as they play an important role in analyzing convergence to stationarity of Markov Chain:  
    \begin{itemize}
        \item Operator norm of $P^{k}:\ell^{2}      (\pi)\rightarrow \ell^{2}(\pi)$ equals one:
            \begin{align}
            \label{eq:P_op_norm_one_ell_2}
            \|P^{k}\|_{\ell^{2}(\pi)}:=\max_{f\in \ell^{2}(\pi):\|f\|_{\ell^{2}(\pi)=1}}\|P^{k}f\|_{\ell^{2}(\pi)}=\|P^{k}\mathbf{1}\|_{\ell^{2}(\pi)}=\|\mathbf{1}\|_{\ell^{2}(\pi)}=1
            \end{align}
        \item Let $\lambda_{1}(P^{k}(P^{*})^{k})\geq \ldots \geq \lambda_{n}(P^{k}(P^{*})^{k})$ be the eigenvalues of $P^{k}(P^{*})^{k}:\ell^{2}(\pi)\rightarrow \ell^{2}(\pi)$ then singular values of transition kernel are defined as $\sigma_{i}(P^{k}):=\sqrt{\lambda_{i}(P^{k}(P^{*})^{k})}$ and have the following variational formulation(Theorem 3.1.1 of \cite{horn1994topics}):
            \begin{align}
            &\sigma_1(P^{k}) = \label{eq:sig_1_at_1} \max \{ \|P^{k}f\|_{\ell^{2}(\pi)} : f \in \ell^{2}(\pi), \|f\|_{\ell^{2}(\pi)} = 1 \}, \quad 
            \text{so } \sigma_1(P^{k}) = \|P^{k}\mathbf{1}\|_{\ell^{2}(\pi)}. \\    
            & \label{eq:sig_2_orth_1_sup} \sigma_2(P^{k}) = \max \{ \|P^{k}f\|_{\ell^{2}(\pi)} : f \perp \mathbf{1}, \|f\|_{\ell^{2}(\pi)} = 1 \}, \quad 
            \text{so } \sigma_2(P^{k}) = \|P^{k} w_2\|_{\ell^{2}(\pi)}
            \text{ for some } w_2 \in V_{\perp 1}. \\ & \nonumber \vdots \\  & \label{eq:sig_n_P_perp} \sigma_n(P^{k}) = \max \{ \|P^{k}f\|_{\ell^{2}(\pi)} : f \perp \mathbf{1}, \ldots, w_{n-1}, \|f\|_{\ell^{2}(\pi)} = 1\}.
            \end{align}        
    \end{itemize}
\end{remark}
\begin{definition}
Issues with conventional definition of spectral gap \cite{hairer2014spectral} A Markov operator $P$ with invariant measure $\pi$ has an $\ell^{2}(\pi)$ spectral gap $1-\beta$ if:
\begin{align}
\label{eq:ell_2}
\beta=\|PU_{\perp 1}\|_{\ell^{2}(\pi)}<1,   
\end{align}
even when spectral gap in $\ell^{2}(\pi)$ sense is $0$ but $\rho(P^{k}U_{\perp 1})<1$ it is possible that $\|P^{k}U_{\perp 1}\|_{\ell^{2}(\pi)} \rightarrow 0$ as discussed in following subsection  
\end{definition}

\subsection{Incompleteness of eigenvalue magnitude for non-asymptotic decay of $\|P^{k}U_{\perp 1}\|_{\ell^{2}(\pi)}$}
\begin{definition}
Given $A \in \mathbb{R}^{d \times d}$ and $x_{0} \in \mathbb{R}^{d}$ consider the following associated $d-$ dimensional linear dynamical system:
\[
x_{k+1}=Ax_{k},
\]
we say that linear dynamical system or matrix $A$ is \emph{stable}  if all the eigenvalues of $A$ are strictly ins unit circle i.e., spectral radius is stricly less than $1$ or mathematically said $\rho(A)<1$.     
\end{definition}
\begin{corollary}
    Stability implies that:
    \begin{align}
        & \nonumber \|A^{k}\|_{\ell^{2}(\mathbb{R}^{d})}\rightarrow 0, \quad \text{consequently}  \quad   \|x_{k}\|_{\ell^{2}(\mathbb{R}^{d})} \rightarrow 0.      
    \end{align}
\end{corollary}
\begin{proof}
    There exists a matrix $J$ comprising of Jordan blocks and a full rank matrix $S \in \mathbb{R}^{d \times d}$ such that: 
     \begin{align}
         & \nonumber A=SJS^{-1} \implies \|A^{k}\|_{\ell^{2}(\mathbb{R}^{d})} \leq \|S\|_{\ell^{2}(\mathbb{R}^{d})}\|S^{-1}\|_{\ell^{2}(\mathbb{R}^{d})}\|J^{k}\|_{\ell^{2}(\mathbb{R}^{d})}          
     \end{align}
     as long as $\rho(A)<1$, asymptotically $\|J^{k}\|_{\ell^{2}(\mathbb{R}^{d})}\rightarrow 0$
\end{proof}
\begin{remark}
    Ergodicity implies that $PU_{\perp 1}$ is a stable matrix as it contains all the eigenvalues of $P$ modulo $1$ and $0$.
\end{remark}
So non-asymptotic rate of convergence for $\|A^{k}\|_{\ell^{2}(\mathbb{R}^{d})}$ requires controlling $\|J^{k}\|_{\ell^{2}(\mathbb{R}^{d})}$ for all finite $k$ along with upper bounds on $\|S\|_{\ell^{2}(\mathbb{R}^{d})}\|S^{-1}\|_{\ell^{2}(\mathbb{R}^{d})}$. Controlling $\|J^{k}\|_{\ell^{2}(\mathbb{R}^{d})}$ requires knowledge of \emph{algebraic} and \emph{geometric} multilicity of eigenvalues of $A$. Position or magnitude of eigenvalues associated to a linear operator $A$ only provides partial information about its' properties (for the ease of exposition, throughout this paper we will assume that $A$ does not have any non-trivial null space). Roughly speaking, algebraic multiplicity of eigenvalues follow from chatacteristic polynomial of the matirx. 
\begin{equation}
    \label{eq:detcharpoly} det(zI-A)= \prod_{i=1}^{K} (z-\lambda_{i})^{m_i},
\end{equation}
where $\lambda_{i}$ are distinct with multiplicity $m_{i}$ and  $\sum_{i=1}^{K} m_i=d$. Since algebraic multiplicity of eigenvalue $\lambda_{i}$ is $m_{i}$ , we denote it by $AM(\lambda_{i})=m_{i}$. Similarly with each $\lambda_{i} \in \mathrm{spec}(A)$, their is an associated set of eigenvectors and dimension of their span corresponds to geometric multiplicity of $\lambda_{i}$ , which we denote by $GM(\lambda_{i})=dim[N(A-\lambda_{i}I)]$. Recall, from linear algebra:
\begin{lemma}
    Let $A \in \mathbb{R}^{d \times d}$. Then $A$ is diagonalizable iff there is a set of $d$ linearly independent vectors, each of which is an eigenvector of $A$. 
\end{lemma}
So, in a situation where $GM(\lambda_{i}) <AM(\lambda_{i}) $, eigenvectors do not span $\mathbb{R}^d$ and one resorts with spanning the underlying state space by direct sum decomposition of $A-$ invariant subspaces (comprising of eigenvector and generalized eigenvectors). 
\begin{remark}[Theorem 4 of \cite{naeem2023spectral}]
    \label{rmk:1stquantdecay}
    For a stable matrix $A$ with $K$ distinct eigenvalues, let discrepancy related to eigenvalue $\lambda_{i}$ be $D_{\lambda_{i}}:= AM(\lambda_{i})-GM(\lambda_{i})$, then 
    \begin{equation}
        \label{eq:qnthdl} \|A^{k}\|_{\ell^{2}(\mathbb{R}^{d})} \leq \|S\|_{\ell^{2}(\mathbb{R}^{d})}\|S^{-1}\|_{\ell^{2}(\mathbb{R}^{d})} \max_{1 \leq i \leq K}  k^{D_{\lambda_{i}}} |\lambda_{i}|^{k} \bigg( \frac{1-|\lambda_{i}|}{1-|\lambda_{i}|^{D_{\lambda_{i}}+1}}\bigg).
    \end{equation}
    Notice: when $D_{\lambda_{i}}=0$ for all $i \in [K]$ then $\|A^{k}\|_{\ell^{2}(\mathbb{R}^{d})} \leq \|S\|_{\ell^{2}(\mathbb{R}^{d})}\|S^{-1}\|_{\ell^{2}(\mathbb{R}^{d})}\rho^{k}(A)$ as expected in diagonalizable case. 
\end{remark}

\section{Main Theoretical Framework}
\label{sec:main_theory_fwk}
\subsection{$\ell^2(\pi)$-Convergence and Variance Decay}

This section studies convergence in the $\ell^2(\pi)$ norm, which quantifies 
how quickly the conditional expectations $\mathbb{E}[f(X_k)|X_0]$ become 
deterministic, converging to $\pi(f)$. This convergence rate directly governs 
the statistical efficiency of MCMC estimators by controlling the decay of 
autocorrelations and the effective sample size. Conventionally, the $\ell^2(\pi)$-distance to stationarity after $k$ steps is defined as:
\begin{align}
    \|p^k - \pi\|_{\ell^{2}(\pi)} := \sup_{f \in \ell^{2}(\pi):\|f-\pi(f)\|_{\ell^{2}(\pi)}=1} \big\|P^{k}f-\pi(f)\big\|_{\ell^{2}(\pi)},
\end{align}

This operator norm controls the worst-case convergence rate across all 
square-integrable functions which are orthogonal to the space of constant functions in $\ell^{2}(\pi)$, and directly governs key practical metrics:

\begin{align}
    \forall f \in \ell^{2}(\pi): \quad & \sqrt{\mathrm{Var}_{X_{0} \sim \pi}\left(\mathbb{E}[f(X_{k})\mid X_{0}]\right)} \leq \|p^k - \pi\|_{\ell^{2}(\pi)} \sqrt{\mathrm{Var}_{\pi}[f]}, \\
    \forall f,g \in \ell^{2}(\pi), k > 0: \quad & \left|\mathrm{Corr}\left(f(X_k), g(X_0)\right)\right| \leq \|p^k - \pi\|_{\ell^{2}(\pi)}.
\end{align}

A fundamental challenge in analyzing non-reversible chains has been the lack of 
computable methods for determining $\|p^k - \pi\|_{\ell^{2}(\pi)}$. Existing 
approaches typically bound this quantity by $\sigma_2(P)^k$, which is only 
tight for reversible dynamics and often fails to detect convergence in 
non-reversible scenarios.Our formulation via $PU_{\perp 1}$ overcomes these 
limitations and provides the first explicitly computable framework for non-reversible chains. As we will show below the following identity $\|p^k - \pi\|_{\ell^{2}(\pi)} = \|(PU_{\perp 1})^k\|_{\ell^2(\pi)}$, which can be explicitly computed(show in following subsection \ref{subsec:iso_pi_n}) as matrix norm of $\|D_{\pi}^{\frac{1}{2}}(PU_{\perp 1})^{k}D_{\pi}^{-\frac{1}{2}}\|_{\ell^{2}(\mathbb{R}^{n})}$, bypassing the limitations of spectral methods that fail for non-reversible chains.
We now analyze the operator $(PU_{\perp 1})^{k}$ as a bounded linear operator 
on $\ell^{2}(\pi)$, establishing its fundamental properties and computational 
tractability.
\begin{theorem}
\label{thm:var_l2_proj}
The $\ell^{2}(\pi)$ distance to stationarity after $k$ steps equals the $\ell^{2}(\pi)$ operator norm of the $k$-th power of the projected transition operator:
\begin{align}
\label{eq:norm_proj=sig_2}
\|P^{k} U_{\perp 1}\|_{\ell^{2}(\pi)}
  = \sup_{\substack{f \in V_{\perp 1}\\ \|f\|_{\ell^{2}(\pi)}=1}}
      \|P^{k} f\|_{\ell^{2}(\pi)}
  = \|p^{k} - \pi\|_{\ell^{2}(\pi)}=\sigma_{2}(P^{k}).
\end{align}
\end{theorem}
\begin{proof}
    \begin{align}
        \|p^{k}-\pi\|_{\ell^{2}(\pi)} & \nonumber :=\sup_{f \in \ell^{2}(\pi):\|f-\pi(f)\|_{\ell^{2}(\pi)}=1} \big\|P^{k}f-\pi(f)\big\|_{\ell^{2}(\pi)}\\ & \nonumber= \sup_{f\in \ell^{2}(\pi):\|U_{\perp 1}f\|_{\ell^{2}(\pi)}=1}\|P^{k}U_{\perp 1}f\|_{\ell^{2}(\pi)} \\ & \nonumber = \sup_{\substack{f \in V_{\perp 1}\\ \|f\|_{\ell^{2}(\pi)}=1}}\|P^{k} f\|_{\ell^{2}(\pi)} \\ & \label{eq:eq_sig_2} = \sigma_{2}(P^{k}),
    \end{align}
where \eqref{eq:eq_sig_2} follows from the equation \eqref{eq:sig_2_orth_1_sup} in remark \ref{rmk: sing_and_mat_norms}. In order to complete the proof note that for any $g$ in space of constant functions w.r.t $\ell^{2}(\pi)$ we have that $\|P^{k}U_{\perp 1}g\|_{\ell^{2}(\pi)}=0$, so $\|P^{k} U_{\perp 1}\|_{\ell^{2}(\pi)}= \sup_{\substack{f \in V_{\perp 1}\\ \|f\|_{\ell^{2}(\pi)}=1}}\|P^{k} f\|_{\ell^{2}(\pi)}$.      
\end{proof}

\subsection{Computational framework via the isomorphism $\ell^{2}(\pi)\cong\ell^{2}(\mathbb{R}^n)$}
\label{subsec:iso_pi_n}

Let $D_\pi=\operatorname{diag}(\pi_1,\dots,\pi_n)$ and define the linear isometry
\[
T:\ell^2(\pi)\to\ell^2(\mathbb{R}^n),\qquad T f = D_\pi^{1/2} f .
\]
Since $T=D_\pi^{1/2}$ is symmetric and invertible we have $T^{-1}=D_\pi^{-1/2}$ and
\[
\langle f,g\rangle_{\ell^2(\pi)}=f^\top D_\pi g = (Tf)^\top (Tg)=\langle Tf,Tg\rangle_{\ell^2(\mathbb{R}^n)}.
\]

Put
\[
\hat P := T P T^{-1} = D_\pi^{1/2} P D_\pi^{-1/2},
\]
so $\hat P$ is the matrix of $P$ in the Euclidean representation induced by $T$.  Below $^T$ denotes Euclidean transpose and $^*$ denotes the adjoint with respect to $\langle\cdot,\cdot\rangle_{\ell^2(\pi)}$; recall the identity
\[
D_\pi P^* = P^T D_\pi.
\]

\begin{theorem}
\label{thm:similar_isomorphism}
For every integer $k\ge1$ the matrices $(\hat P^{k})^{T}\hat P^{k}$ and $(P^{k})^{*}P^{k}$ are similar. Consequently, the singular values of $P^{k}$ acting on $\ell^2(\pi)$ equal the singular values of the $n\times n$ matrix $\hat P^{k}$ acting on $\ell^2(\mathbb{R}^n)$. 
\end{theorem}

\begin{proof}
Using $T^T=T$ and $T^{-1}=D_\pi^{-1/2}$ we compute
\[
(\hat P^{k})^{T}\hat P^{k}
= (T P^{k} T^{-1})^{T}(T P^{k} T^{-1})
= T^{-1}(P^{k})^{T} T T P^{k} T^{-1}
= D_\pi^{-1/2} (P^{k})^{T} D_\pi\, P^{k} D_\pi^{-1/2}.
\]
By the adjoint identity $(P^{k})^{T} D_\pi = D_\pi (P^{k})^{*}$ (equivalently $(P^{k})^*=(P^*)^k$), the last display equals
\[
D_\pi^{-1/2} D_\pi (P^{k})^{*}P^{k} D_\pi^{-1/2} = D_\pi^{1/2} (P^{k})^{*}P^{k} D_\pi^{-1/2}
= T\big((P^{k})^{*}P^{k}\big)T^{-1}.
\]
Thus $(\hat P^{k})^{T}\hat P^{k}$ is similar to $(P^{k})^{*}P^{k}$, and hence they have the same eigenvalues. Since singular values are the square roots of these eigenvalues, the singular values of $\hat P^{k}$ (in the Euclidean sense) coincide with the singular values of $P^{k}$ in $\ell^2(\pi)$.
\end{proof}
\begin{remark}
In case of  $\ell^{2}(\pi)$ convergence we are primarily interested in decay of $\|P^{k}U_{\perp 1}\|_{\ell^{2}(\pi)}$ which can now be computed scientifically as $\big\|D_{\pi}^{1/2} P^{k} U_{\perp 1} D_{\pi}^{-1/2}\big\|_{\ell^{2}(\mathbb{R}^n)}$, based on preceding Theorem \ref{thm:similar_isomorphism}   
\begin{align}
\label{eq:sigma2-compute}
\sigma_2(P^{k}) \;=\; \|P^{k} U_{\perp 1}\|_{\ell^{2}(\pi)}
\;=\; \big\|D_{\pi}^{1/2} P^{k} U_{\perp 1} D_{\pi}^{-1/2}\big\|_{\ell^{2}(\mathbb{R}^n)}
\end{align}    
\end{remark}

\subsection{$\chi^{2}$ Convergence, Pointwise Submultiplicativity and Hypocoercivity}

This section establishes powerful results for $\chi^2$-convergence, including the first general proof of pointwise submultiplicativity for non-reversible chains and explicit constants for hypocoercivity. The $\chi^2$-divergence provides a stronger notion of convergence than total variation, particularly useful for quantifying relative error in MCMC estimation.

\begin{definition}[$\chi^2$-mixing time]
    For $\epsilon > 0$, the $\chi^2$-mixing time is defined as:
    \begin{align}
        \tau_{\chi^{2}}(\epsilon):=\inf\left\{t>0: \max_{i \in [n]} \chi_{i}^{2}(t) \leq \epsilon \right\},
    \end{align}
    where $\chi_{i}^{2}(t)=\sum_{j=1}^{n}\frac{(p^{t}(i,j)-\pi_{j})^{2}}{\pi_{j}}$, where $p^{t}(i,j)$ is the probability of being in state $j$ after $t$ time steps starting from i. Now notice that $\chi_{i}^{2}(0)=\sum_{j=1}^{n}\frac{(p^{0}(i,j)-\pi_{j})^{2}}{\pi_{j}}=\frac{(1-\pi_{i})^{2}}{\pi_{i}}+\sum_{j \neq i}\pi_{j}=\frac{(1-\pi_{i})^{2}}{\pi_{i}}+(1-\pi_{i})=\frac{1-\pi_{i}}{\pi_{i}}.$ 
\end{definition}

\begin{definition}[Global $\chi^2$-divergence and singular values]
    The averaged $\chi^2$-divergence $\vec{\chi^{2}}(k):=[\chi_{1}^{2}(k),\ldots,\chi_{n}^{2}(k)]$ satisfies:
    \begin{align}
        \langle \vec{\chi^{2}}(k),\mathbf{1}\rangle_{\pi} = \sum_{i}\pi_{i}\chi_{i}^{2}(k) = \mathrm{Tr}(P^{k}(P^{k})^{*})-1 = \sum_{j=2}^{n}\sigma_{j}^{2}(P^{k}).
    \end{align}
    This reveals that the \emph{average} $\chi^2$-divergence equals the sum of squared non-trivial singular values, providing a spectral characterization of global convergence.
\end{definition}

\subsubsection*{Key Theoretical Advance: Pointwise Submultiplicativity}

\begin{theorem}[Pointwise $\chi^2$ submultiplicativity]
    For any initial state $i \in [n]$ and times $k,t \in \mathbb{N}$:
    \begin{align}
        \chi_{i}^{2}(k+t) \leq \sigma_{2}^{2}(P^{t}) \chi_{i}^{2}(k).
    \end{align}
\end{theorem}
\begin{proof}
    Let $\vec{\chi_{i}}(k):=(P^{k})^{*}D_{\pi}^{-1}e_{i}-1$ then notice that for each $i \in [n], \langle \vec{\chi_{i}}(k),\mathbf{1}\rangle_{\pi}=0, \chi_{i}^{2}(k)=\|\vec{\chi_{i}}(k)\|_{\ell^{2}(\pi)}^{2}$. Furthermore, notice that for any $t,k \in \mathbb{N}$, we have that $\vec{\chi_{i}}(t+k)=(P^{t})^{*}\vec{\chi_{i}}(k) \in V_{\perp 1}$ orthogonal to the space of constant functions in $\ell^{2}(\pi)$ i.e. $\langle (P^{t})^{*}\vec{\chi_{i}}(k),\mathbf{1}\rangle_{\pi}=0$ and

    \begin{align}
        \|\vec{\chi_{i}}(t+k)\|_{\ell^{2}(\pi)}^{2}& \nonumber =\langle (P^{t})^{*} \vec{\chi_{i}}(k) ,(P^{t})^{*} \vec{\chi_{i}}(k)\rangle_{\pi}\\ &=\langle P^{t}(P^{t})^{*} \vec{\chi_{i}}(k) , \vec{\chi_{i}}(k)\rangle_{\pi} \\ & \label{eq:cauchy_schwarz} \leq \sqrt{\|P^{t}(P^{t})^{*} \vec{\chi_{i}}(k)\|_{\ell^{2}(\pi)}^{2}\|\vec{\chi_{i}}(k)\|_{\ell^{2}(\pi)}^{2}} \\ & \label{eq: orth_constant_sub} \leq \sqrt{\sigma_{2}^{2}(P^{t})\|(P^{t})^{*} \vec{\chi_{i}}(k)\|_{\ell^{2}(\pi)}^{2}\|\vec{\chi_{i}}(k)\|_{\ell^{2}(\pi)}^{2}}   \\ & \nonumber = \sqrt{\sigma_{2}^{2}(P^{t})\|\vec{\chi_{i}}(t+k)\|_{\ell^{2}(\pi)}^{2}\|\vec{\chi_{i}}(k)\|_{\ell^{2}(\pi)}^{2}} \\ & \nonumber =\sigma_{2}(P^{t})\|\vec{\chi_{i}}(t+k)\|_{\ell^{2}(\pi)}\|\vec{\chi_{i}}(k)\|_{\ell^{2}(\pi)},
    \end{align}
    where inequality \ref{eq:cauchy_schwarz} follows from Cauchy-Schwarz, inequality \ref{eq: orth_constant_sub} follows from the fact that $(P^{t})^{*}\vec{\chi_{i}}(k) \in V_{\perp 1}$ and second last equality follows from $\vec{\chi_{i}}(t+k)=(P^{t})^{*}\vec{\chi_{i}}(k)$ and the result follows.  
\end{proof}

\begin{corollary}[Worst-case non-asymptotic $\chi^2$ bounds]
    \label{corr:submcorr}
    For any initial state $i$ and time $t$: notice that $\chi_{i}^{2}(t) \leq \sigma_{2}^{2}(P^{t})\left(\frac{1-\pi_{i}}{\pi_{i}}\right)$ and an upper bound on chi-squared distance to stationarity follows:
    \begin{align}
        \label{eq:chi_square_ubd_non_asymp}
      \chi^{2}(t) \leq \sigma_{2}^{2}(P^{t})\bigg(\frac{1-\pi_{\min}}{\pi_{min}}\bigg).        
    \end{align}
\end{corollary}    
\begin{theorem}
    \label{thm:asymp_conv_chi}
    [Worst-case non-asymptotic $\chi^2$ bounds for an $n$-dimensional ergodic Markov Chain]: Let $\{\lambda_{i}\}_{i \in [K]}$ be the distinct non-zero eigenvalues of $PU_{\perp1}$ with discrepancies $D_{\lambda_{i}}$ between their algebraic and geometric multiplicities. Let $D_{\pi}^{\frac{1}{2}}PU_{\perp1}D_{\pi}^{-\frac{1}{2}}=SJ_{\perp1}S^{-1}$ be its Jordan canonical form decomposition. Then for all $k \in \mathbb{N}$:
    \begin{align}
        \chi^{2}(k) \leq \left(\frac{1-\pi_{\min}}{\pi_{\min}}\right) \kappa(S)^2 \max_{1 \leq i \leq K} \left[ k^{2D_{\lambda_{i}}} |\lambda_{i}|^{2k} \left( \frac{1-|\lambda_{i}|}{1-|\lambda_{i}|^{D_{\lambda_{i}}+1}}\right)^{2} \right],
    \end{align}
    where $\kappa(S) = \|S\|_{\ell^{2}(\mathbb{R}^{n})}\|S^{-1}\|_{\ell^{2}(\mathbb{R}^{n})}$ is the condition number of the similarity transformation.
\end{theorem}
\begin{proof}
The result follows by combining the pointwise $\chi^2$ bound: $\chi^2(k) \leq \left(\frac{1-\pi_{\min}}{\pi_{\min}}\right) \|P^k U_{\perp 1}\|_{\ell^2(\pi)}^2$, The similarity transformation: $\|P^k U_{\perp 1}\|_{\ell^2(\pi)} = \|D_{\pi}^{1/2} P^k U_{\perp 1} D_{\pi}^{-1/2}\|_{\ell^2(\mathbb{R}^n)}$ and setting $A=D_{\pi}^{1/2} P^k U_{\perp 1} D_{\pi}^{-1/2}$ in Remark \ref{rmk:1stquantdecay}.
\end{proof}

\begin{corollary}
\label{cor:diag_chi_sq}
For ergodic chains that are diagonalizable (i.e.  $D_{\lambda_i} = 0$ for all $i\in [K]$), we have that:
\begin{align}
    \label{eq:chi_diag_case}
    \chi^{2}(k) \leq \left(\frac{1-\pi_{\min}}{\pi_{\min}}\right) \kappa(S)^2 \rho^{2k}\big(D_{\pi}^{1/2} P^k U_{\perp 1} D_{\pi}^{-1/2}\big)
\end{align}
\end{corollary}
\begin{definition}[Discrete-time hypocoercivity]
    A Markov chain exhibits hypocoercivity if there exist constants $C<\infty$ and $\lambda \in (0,1)$ such that:
    \begin{align}
        \label{eq:hyp_constants}
        \sup_{f \in \ell^{2}(\pi):\|f-\pi(f)\|_{\ell^{2}(\pi)}=1} \|P^{k}f-\pi(f)\|_{\ell^{2}(\pi)} \leq C \lambda^{k}.
    \end{align}
\end{definition}
Previous works \cite{baxendale2005renewal,villani2006hypocoercivity} provided existential results under ergodicity assumption

\begin{remark}[Hypocoercivity constants under diagonalizability assumption]
\label{rmk:hypocoercivity-constants} 
For diagonalizable ergodic Markov chain with transition matrix $P$ and stationary distribution $\pi$, let $\rho = \rho(D_{\pi}^{1/2} P U_{\perp 1} D_{\pi}^{-1/2})$ and $\kappa(S)$ be the condition number of the similarity transformation to Jordan form. Then $\lambda =\rho$ and $C=\kappa(S)$.
\end{remark}

\subsection{Relaxation Times and Neumann Series of $PU_{\perp 1}$}
\label{sec:dev_avg_relaxation_time}
Recall that generator of an $n$- dimensional Markov Chain $L:=I_{n}-P$. Since $L\mathbf{1}=0$, we see that smallest singular value of $L$: $\sigma_{n}(L)=\|L1\|_{\ell^{2}(\pi)}=0$ and by variational formulation of singular values as in remark \ref{rmk: sing_and_mat_norms}: $\sigma_{n-1}(L)=\inf_{f \in V_{\perp 1}\cap\|f\|_{\ell^{2}(\pi)}=1}\|Lf\|_{\ell^{2}(\pi)}$. It's often infeasible to directly compute $\pi(f)$, either due to complexity of $f$ or $\pi$ only being known up to a proportionality constant. So practitioners run a Markov chain with $\pi$ as stationary measure and after it mixes to $\pi$: we ask how good of an estimator are the ergodic averages from Markov Chain $\frac{1}{N}\sum_{i=0}^{N-1}f(x_{i})$, for $\pi(f)$? A baseline for reference is independent and identically distributed $y_{i} \thicksim \pi$ with variance:
$\mathrm{Var}(\tfrac{1}{N}\sum_{i=0}^{N-1}f(y_{i}))=\tfrac{N \mathrm{Var}_{\pi}(f)}{N^{2}}=\tfrac{1}{N}$ but when you have samples from stationary chain:
\begin{remark}[Chatterjee's Relaxation Time]
\label{thm: var_erg_avgs}
\cite[Theorem 1.2]{chatterjee2023spectral} establishes that for any $N > 0$ and $f:\mathrm{Var}_{\pi}(f)=1$
\begin{align}
    \label{eq:relax_chatterjee}
    \mathrm{Var}_{x_{0}\thicksim \pi}\bigg(\frac{1}{N}\sum_{i=0}^{N-1}f(x_{i})\bigg) \leq  \frac{4\tau_{rel}}{N},
\end{align}
where $\tau_{rel}:=\frac{1}{\sigma_{n-1}(L)}$ is defined as the \emph{relaxation time}. 
\end{remark}
Factor of $4 \tau_{rel}$ compared to i.i.d sampling must be originating from autocorrelations at various time lags as:
\begin{align}
    \mathrm{Var}_{\pi}\Big(\tfrac{1}{N}\sum_{i=0}^{N-1}f(x_{i})\Big)& \nonumber =\tfrac{1}{N^{2}}\sum_{i=0}^{N-1}\Big[\mathrm{Var}_{\pi}(f(x_{i}))+2\sum_{j>i}\mathrm{Cov}_{\pi}(f(x_{i}),f(x_{j}))\Big] \\ & \label{eq:stationarity}= \tfrac{1}{N^{2}}\Big[N+2\sum_{i=1}^{N-1}(N-i)\mathrm{Cov}_{\pi}(f(x_{0}),f(x_{i}))\Big] \\ & \label{eq:corr_decay} \leq \tfrac{1}{N}+\Big(\tfrac{2}{N^{2}}\Big)\sum_{i=1}^{N-1}(N-i)\|P^{i}U_{\perp 1}\|_{\ell^{2}(\pi)} 
\end{align}
where the equation \eqref{eq:stationarity} follows from stationarity and equation \eqref{eq:corr_decay} follows form Cauchy-Schwarz. To formalize the relation between the relaxation time \(t_{\mathrm{rel}}=1/\sigma_{n-1}(L)\) and autocorrelations at all lags, we have to look at the inverse of generator.
\begin{theorem}
    $L$ is a bijection on $V_{\perp 1}$, as $V_{\perp 1}$ is $L$-invariant and kernel of $L$ lies in the space of constant functions in $\ell^{2}(\pi)$, by inverse mapping theorem $L^{-1}:V_{\perp 1} \rightarrow V_{\perp 1}$ exists, with $\ell^{2}(\pi)$ norm equal to relaxation time.  
    \begin{align}
        \|L^{-1}\|_{V_{\perp 1}}=\frac{1}{\sigma_{n-1}(L)}
    \end{align}
\end{theorem}
\begin{proof}
Bijection of $L$ on $V_{\perp 1}$ is easy to check and we will focus on proving norm equality. $\|\cdot\|_{V_{\perp 1}}$ is essentially $\ell^{2}(\pi)$ norm on the elements of $V_{\perp 1}$, so for any $f \in V_{\perp 1} \cap \|f\|_{\ell^{2}(\pi)}=1$   
\begin{align}
    & \nonumber 1=\|L^{-1}Lf\|_{V_{\perp 1}}\leq \|L^{-1}\|_{V_{\perp 1}}\|Lf\|_{V_{\perp 1}} \\ & \frac{1}{\inf_{f \in V_{\perp 1}\cap \|f\|_{\ell^{2}(\pi)}=1}\|Lf\|_{V_{\perp 1}}} \leq \|L^{-1}\|_{V_{\perp 1}} \\ & \nonumber
    \frac{1}{\sigma_{n-1}(L)} \leq \|L^{-1}\|_{V_{\perp 1}},
\end{align}
and similarly for any $f \in V_{\perp 1}\cap \|f\|_{\ell^{2}(\pi)}=1$
\begin{align}
    & \nonumber 1=\|LL^{-1}f\|_{V_{\perp 1}}\geq \sigma_{n-1}(L) \|L^{-1}f\|_{V_{\perp 1}} \\ & \nonumber \|L^{-1}\|_{V_{\perp 1}}:=\sup_{f \in V_{\perp 1}\cap \|f\|_{\ell^{2}(\pi)}=1}\|L^{-1}f\|_{V_{\perp 1}} \leq \frac{1}{\sigma_{n-1}(L)},
\end{align}
and the result follows.
\end{proof}
\begin{theorem}[Relaxation Time and Autocorrelation Structure]
The relaxation time $t_{\mathrm{rel}} = 1/\sigma_{n-1}(L)$ of an ergodic Markov Chain $\rho(PU_{\perp 1})<1$; fundamentally captures the cumulative effect of autocorrelations across all time lags $k$ and can be explicitly computed as:
   \begin{align}
       \Bigg\|D_{\pi}^{\frac{1}{2}}\Big(\sum_{k=0}^{\infty}P^{k}\Big)U_{\perp 1}D_{\pi}^{-\frac{1}{2}}\Bigg\|_{\ell^{2}(\mathbb{R}^{n})}=\frac{1}{\sigma_{n-1}(L)}
   \end{align}
\end{theorem}
\begin{proof}
    
Consider the modified generator $L_{\perp 1}:=I-PU_{\perp 1}$, as $\rho(PU_{\perp 1})<1$ inverse exists and is given by Neumann Series:
\begin{align}
    & \nonumber L^{-1}_{\perp 1}=(I-PU_{\perp 1})^{-1}=\sum_{k=0}^{\infty}(PU_{\perp 1})^{k}, \hspace{10pt} \textit{easy to verify that} \\ & L^{-1}_{\perp 1}(I-PU_{\perp 1})  = (I+\sum_{k=1}^{\infty}P^{k}U_{\perp 1})(I-PU_{\perp 1})=I-PU_{\perp 1}+\sum_{k=1}^{\infty}P^{k}U_{\perp 1}-\sum_{k=2}^{\infty}P^{k}U_{\perp 1}=I
\end{align}  
Notice that $L=L_{\perp 1}$ on $V_{\perp 1}$, and we can define on $V_{\perp 1}$, inverse of the generator $L^{-1}=I+\sum_{k=1}^{\infty}P^{k}U_{\perp 1}$. Also notice that $\|L^{-1}\|_{V_{\perp 1}}=\sup_{f \in V_{\perp 1}\cap\|f\|_{\ell^{2}(\pi)}=1}\|(I+\sum_{k=1}^{\infty}P^{k}U_{\perp 1})f\|_{\ell^{2}(\pi)}=\sup_{f:\|f\|_{\ell^{2}(\pi)}=1}\|\sum_{k=0}^{\infty}P^{k}U_{\perp 1}f\|_{\ell^{2}(\pi)}=:\|\sum_{k=0}^{\infty}P^{k}U_{\perp 1}\|_{\ell^{2}(\pi)}$ and the claim follows after including similarity transformations. \end{proof}

\begin{table}[h!]
\centering
\begin{tabular}{|l|c|}
\hline
\textbf{Functional Form} $\ell^{2}(\pi)$ & \textbf{Matrix Representation} \\ 
\hline
$P^{*}$ & $D_{\pi}^{-1} P^{T} D_{\pi}$ \\ 
\hline
$\sigma_{2}(P^{k})=\big\|(PU_{\perp 1})^{k}\big\|_{\ell^{2}(\pi)} $ & $\|D_{\pi}^{\frac{1}{2}}(PU_{\perp 1})^{k}D_{\pi}^{-\frac{1}{2}}\|_{\ell^{2}(\mathbb{R}^n)}$ \\ 
\hline
$\rho(PU_{\perp 1}) $ & $\rho(D_{\pi}^{\frac{1}{2}}PU_{\perp 1}D_{\pi}^{-\frac{1}{2}})$ \\
\hline
$\sigma_{j}(P^{k}), j \geq 2$ & $\lambda_{j}\Big(\big(D_{\pi}^{\frac{1}{2}}P^{k}D_{\pi}^{-\frac{1}{2}}\big)\big(D_{\pi}^{\frac{1}{2}}P^{k}D_{\pi}^{-\frac{1}{2}}\big)^{T}\Big)$ \\
\hline
$\frac{1}{\sigma_{n-1}(L)}$ & $ \bigg\|D_{\pi}^{\frac{1}{2}}\Big(\sum_{k=0}^{\infty}P^{k}\Big)U_{\perp 1}D_{\pi}^{-\frac{1}{2}}\bigg\|_{\ell^{2}(\mathbb{R}^{n})}$ \\
\hline
\end{tabular}
\caption{Functional Form v.s. Matrix representations.}
\label{tab:func_vs_mat}
\end{table}

\section{Case Studies}
\label{sec:case_studies}
\paragraph{Three state non-reversible Markov Chain}
\begin{remark}
\label{rmk:three_state_non_rev}
Consider the Example 4.5.1 of three-state non-reversible Markov Chain given in, \cite{choi2017analysis} with $p(1,2)=1, p(2,3)=1, p(3,1)=\tfrac{1}{2}$ and $p(3,2)=\tfrac{1}{2}$,  
we provide the sharpest convergence bound:
\begin{align}
    \|p^{k}-\pi\|_{TV} \leq O\bigg(\sqrt{\tfrac{1}{2}}^{k} \hspace{2pt} \bigg),
\end{align}
which is much tighter than the bound from \cite{paulin2015concentration} with constant $\sqrt{0.75}$ and \cite{choi2017analysis} with constant $0.849$.
\end{remark}
\begin{proof}
Stationary distribution corresponding to the transition kernel $\pi=[\tfrac{1}{5},\tfrac{2}{5},\tfrac{2}{5}]$ so $\pi_{\min}=\tfrac{1}{5},\tfrac{1-\pi_{\min}}{\pi_{\min}}=4$ and :
\[
P = \begin{bmatrix} 0 & 1 & 0 \\ 0 & 0 & 1 \\ \tfrac{1}{2} & \tfrac{1}{2} & 0 \end{bmatrix}, \quad
U_{\perp 1} = \begin{bmatrix} \tfrac{4}{5} & -\tfrac{2}{5} & -\tfrac{2}{5} \\ -\tfrac{1}{5} & \tfrac{3}{5} & -\tfrac{2}{5} \\ -\tfrac{1}{5} & -\tfrac{2}{5} & \tfrac{3}{5} \end{bmatrix}, \quad
D_{\pi}^{\frac{1}{2}}PU_{\perp 1}D_{\pi}^{-\frac{1}{2}}= \begin{bmatrix} -\tfrac{1}{5} & \tfrac{3\sqrt{2}}{10} & -\tfrac{\sqrt{2}}{5} \\ -\tfrac{\sqrt{2}}{5} & -\tfrac{2}{5} & \tfrac{3}{5} \\ \tfrac{3\sqrt{2}}{10} & \tfrac{1}{10} & -\tfrac{2}{5} \end{bmatrix}  
, \|D_{\pi}^{\frac{1}{2}}PU_{\perp 1}D_{\pi}^{-\frac{1}{2}}\|_{\ell^{2}(\mathbb{R}^{n})}=1,\]
although does not have any \emph{spectral gap} but $\mathrm{spec}(D_{\pi}^{\frac{1}{2}}PU_{\perp 1}D_{\pi}^{-\frac{1}{2}})=[0,\tfrac{1}{2}+\tfrac{j}{2},\tfrac{1}{2}-\tfrac{j}{2}]$(where $j:=\sqrt{-1}$) which means that spectral radius $\rho(D_{\pi}^{\frac{1}{2}}PU_{\perp 1}D_{\pi}^{-\frac{1}{2}})=\sqrt{\tfrac{1}{2}}$, so  asymptotically $\|P^{k}U_{\perp 1}\|_{\ell^{2}(\pi)}=\tfrac{1}{\sqrt{2}^{k}}$ and consequently $\chi^{2}(k) \leq \tfrac{4}{2^{k}}$. As an application of cauchy-schwarz we can conclude that $\|p^{k}-\pi\|_{TV} \leq \frac{1}{2}\sqrt{\chi^{2}(k)}\leq \sqrt{\tfrac{1}{2}}^{k}$.   
\end{proof}
\paragraph{Diaconis Holmes and Neal, Lifted/Momentum based sampler \cite{diaconis2000analysis}} 
\label{para:diac_algo}
 Let $\pi(x)$ be a strictly positive distribution on $x \in \{1,2,\ldots,n\}$, introduce a momentum variable $z \in (-1,+1)$: associated stationary measure on the lifted space is $\hat{\pi}(z,x):=\frac{\pi(x)}{2}$(ensures that marginal along $x$ is same as $\pi(x)$) and the sampling method proposed is as follows :
    \begin{enumerate}
        \item From $(z,x)$, try to move to $(-z,x+z)$ via standard Metropolis step using acceptance probability
        \begin{equation}
            \nonumber \alpha_{(z,x)\rightarrow(-z,x+z)} = \min \Big\{1,\frac{\hat{\pi}(-z,x+z)}{\hat{\pi}(z,x)}\Big\}=\min\Big\{1,\frac{\pi(x+z)}{\pi(x)}\Big\}.
        \end{equation}
        if $(x+z) \not \in \{1,2,\ldots,n\}$ then set $\alpha_{(z,x)\rightarrow(-z,x+z)}=0$. 
        \begin{align}
            & \nonumber (z',x') = (-z,x+z) \hspace{5pt} \textit{w.p.} \hspace{5pt} \alpha_{(z,x)\rightarrow(-z,x+z)} \\ & \nonumber (z',x')=(z,x)  \hspace{5pt} \textit{w.p.} \hspace{5pt} 1-\alpha_{(z,x)\rightarrow(-z,x+z)}
        \end{align}
        \item With probability $1-\theta$, the chain moves to $(-z',x')$ and w.p $\theta$ stays at $(z',x')$.   
    \end{enumerate}
Although our matrix based framework, when combined with strong compute capabilities can explicitly compute distance to stationarity for all finite times $k$ and reveal explicit dependence on dimension $n$, using python library like \emph{Sympy}. However, even with minimalistic computations we can derive explicit dimensional dependencies for $\chi^{2}$ convergence of momentum based sampler. Computational framework is as follows:

\begin{algorithm}[H]
\caption{Dimensional dependence analysis for $\chi^2$ convergence of momentum-based sampler (Diagonalizable dynamics)}
\label{alg:dim_dependence}
\KwIn{Dimension range $n \in \{50, 51, \ldots, 150\}$ and target distribution $\pi$ on $n$ points}
\KwOut{Estimated parameters $c$, $\alpha$, $C$, $\beta$, and the $\chi^2$ convergence bound}

\For{$n \leftarrow 50$ \KwTo $150$}{
    Construct transition kernel 
    $P \in \mathbb{R}^{2n \times 2n}$ corresponding to the lifted sampler (see \ref{para:diac_algo})\;

    \tcp{Compute the smallest stationary probability (omit for uniform $\pi$)}
    $\pi_{\min}(n) \leftarrow \min_i \pi(i)$\;

    \tcp{Form the lifted stationary distribution over $2n$ states}
    $\hat{\pi} \leftarrow [\tfrac{\pi_{1}}{2}, \tfrac{\pi_{1}}{2}, \tfrac{\pi_{2}}{2}, \tfrac{\pi_{2}}{2}, \ldots, \tfrac{\pi_{n}}{2}, \tfrac{\pi_{n}}{2}]$\;

    \tcp{Construct diagonal matrices of $\sqrt{\hat{\pi}}$ and its inverse}
    $D_{\hat{\pi}}^{1/2} \leftarrow \operatorname{diag}(\sqrt{\hat{\pi}_1}, \ldots, \sqrt{\hat{\pi}_{2n}})$\;
    $D_{\hat{\pi}}^{-1/2} \leftarrow \operatorname{diag}(1/\sqrt{\hat{\pi}_1}, \ldots, 1/\sqrt{\hat{\pi}_{2n}})$\;

    \tcp{Construct projection matrix for the subspace orthogonal to the space of constant functions}
    $U_{\perp 1} \leftarrow I_{2n} - \mathbf{1}\hat{\pi}^{\top}$\;

    \tcp{Compute matrix similar to $\ell^{2}(\pi)$ operator $PU_{\perp 1}$}
    $M(n) \leftarrow D_{\hat{\pi}}^{1/2} P U_{\perp 1} D_{\hat{\pi}}^{-1/2}$\;

    \tcp{Compute the spectral radius of $M(n)$}
    $\rho(n) \leftarrow \rho(M(n))$\;
}

\tcp{Fit dimension-dependence of spectral gap}
Transform: $y_1(n) \leftarrow \log(1 - \rho(n))$, $x(n) \leftarrow \log(n)$\;
Perform linear regression: $y_1(n) = \log(c) - \alpha \cdot x(n)$\;

\tcp{Fit dimension-dependence of minimal stationary probability (if non-uniform)}
Transform: $y_2(n) \leftarrow \log(\pi_{\min}(n))$\;
Perform linear regression: $y_2(n) = \log(C) + \beta \cdot x(n)$\;

\tcp{Compute the resulting $\chi^2$ convergence bound}
$\chi^{2}(k) \leq \left(\frac{2}{C n^{\beta}} - 1\right) \left(1 - \frac{c}{n^{\alpha}}\right)^{2k}$\;

\Return $c$, $\alpha$, $C$, $\beta$, and the convergence bound\;
\end{algorithm}
Based on the preceding algorithm we conclude: 
\begin{remark}    
When desired distribution is uniform i.e., $\pi(x)=\frac{1}{n}$ for $x\in \{1,2,\ldots,n\}$, using the momentum based sampler with $\theta=\frac{1}{n}$ then our analysis show that for sufficiently large $k$
\begin{align}
    \chi^{2}(k) \leq (2n-1)(1-\tfrac{1}{n})^{2k}
\end{align}
which is consistent with Theorem 2 of \cite{diaconis2000analysis} i.e., $n\log{n}$ steps suffices for $\chi^2$ convergence as opposed to $n^{2}$ steps required by Metropolis-Hasting algorithm. Due to uniform nature of stationary distribution it is     
\end{remark}

\begin{remark}
    Momentum based sampler for $\pi(x)\propto 2|x-\frac{n}{2}|+2$ is diffusive, with flipping probability $\theta=\frac{1}{n}$, $n^{1.969}\log{n^{1.956}}$ iterations are sufficient for $\chi^{2}$ convergence. See figures \ref{fig:even_v}, \ref{fig:odd_v} for dimensional dependence of $\rho(PU_{\perp 1})$ and table \ref{tab:asymp_conv_rates} for exact estimates. We have verified that all the eigenvalues are distinct, hence dynamics are diagonalizable and $\chi^{2}$ bounds are correct modulo condition number of similarity transform.     
\end{remark}
\begin{figure}[h!]
\centering
\begin{subfigure}{\textwidth}
    \centering
    \includegraphics[height=7cm]{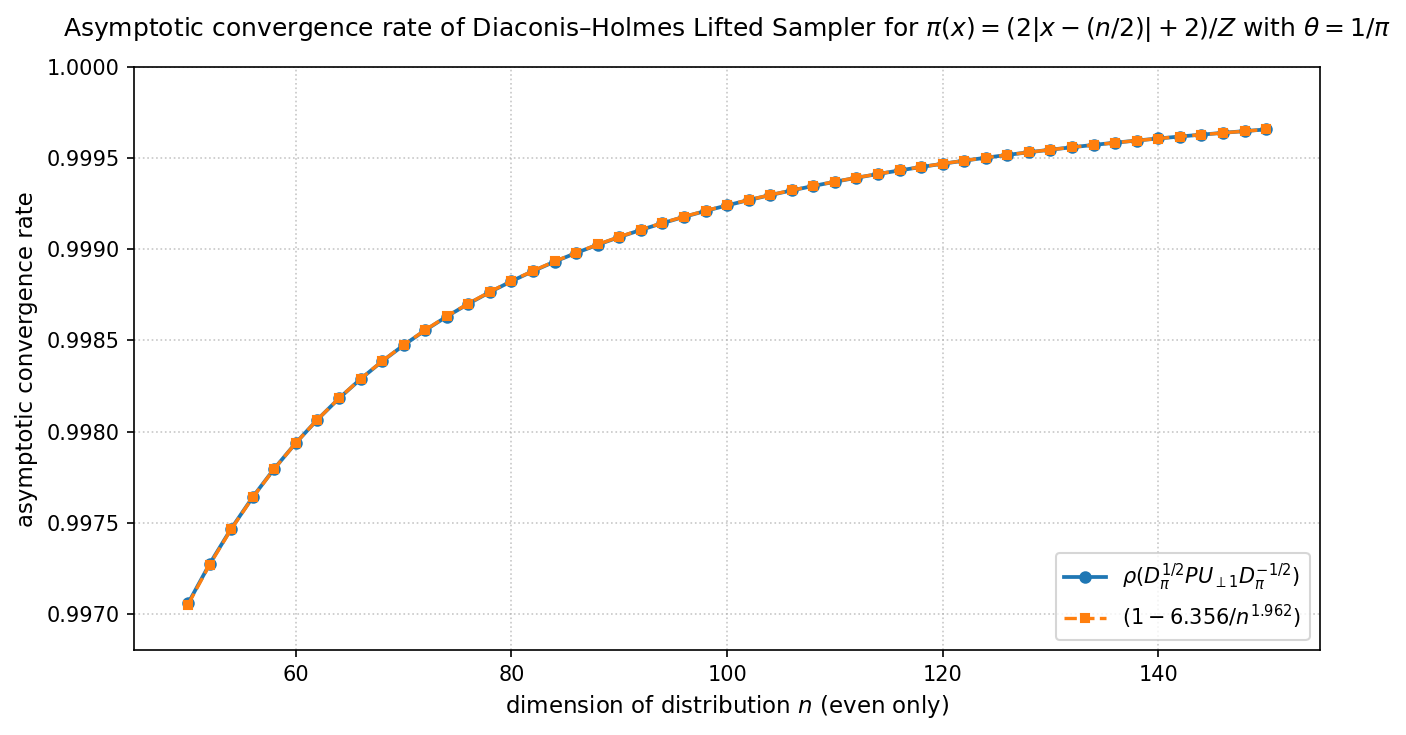}
    \caption{Spectral radius of projected transition matrix v.s dimension $n$(even values)}
    \label{fig:even_v}
\end{subfigure}

\vspace{0.2cm} 

\begin{subfigure}{\textwidth}
    \centering
    \includegraphics[height=7cm]{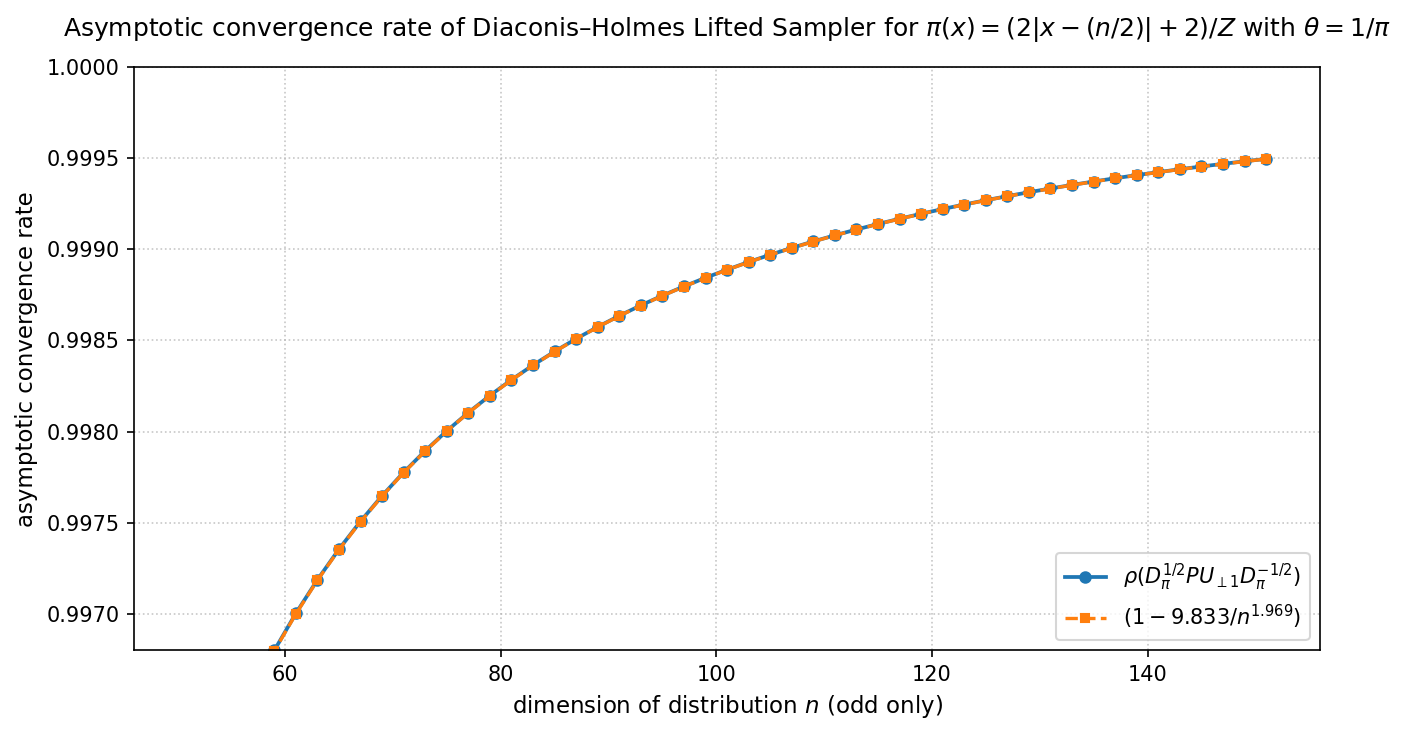}
    \caption{Spectral radius of projected transition matrix v.s dimension $n$(odd values)}
    \label{fig:odd_v}
\end{subfigure}
\caption{Asymptotic convergence rate of Diaconis-Holmes Lifted Sampler for V-shaped distribution}
\label{fig:convergence_rates}
\end{figure}

\bibliographystyle{imsart-nameyear}  
\bibliography{aap-sample}
\end{document}